\documentclass[12pt]{article}
\usepackage{amsmath}
\usepackage{amsfonts}
\usepackage{amssymb}
\usepackage{amsthm}
\usepackage{stackrel}

\usepackage{color}

\usepackage{figsize}

\usepackage{graphics}

\newcommand{\mc}{\mathcal }

\newtheorem{thm}{Theorem}[section]

\newcommand{\go}[1]{\mathfrak{#1}}

\newcommand{\R}{{\rm I}\kern-0.18em{\rm R}}
\newcommand{\1}{{\rm 1}\kern-0.25em{\rm I}}
\newcommand{\E}{{\rm I}\kern-0.18em{\rm E}}
\newcommand{\p}{{\rm I}\kern-0.18em{\rm P}}

\makeatletter

\def\@fnsymbol#1{\ensuremath{\ifcase#1\or a\or b\or c\or d\or \e\or f\or *\dagger 	\or \ddagger\ddagger \else\@ctrerr\fi}}

\title{Characterization of multivariate distributions by means of univariate one}
\author{Lev B. Klebanov\footnote{Department of Probability and Mathematical Statistics, Charles University, Prague, Czech Republic, e-mail: lev.klebanov@mff.cuni.cz} and Irina Volchenkova\footnote{Czech Technical University in Prague, Prague, Czech Republic}}
\date{}

\begin{document}
\maketitle

\begin{abstract}
The aim of this paper is to show a possibility to identify multivariate distribution by means of specially constructed one-dimensional random variable. We give some inequalities which may appear to helpful for a construction of multivariate two-sample tests. 

\end{abstract}

\section{Inequalities for multivariate distributions}\label{sec1}
\setcounter{equation}{0}

Let $X$ and $Y$ be two independent random vectors in Euclidean space $\R^d$ or in separable Hilbert space $\go H$. Suppose that $X^{\prime}$ $Y^{\prime}$ are independent copies of $X$ and $Y$ correspondingly, that is $X^{\prime}\stackrel{d}{=}X$ and  $Y^{\prime}\stackrel{d}{=}Y$. The simbol $\stackrel{d}{=}$ here denotes the equality in distribution. Consider two random variables $\varepsilon$ and $\delta$ taking two values $0$ and $1$ with equal probabilities. In addition, let us suppose that the objects $X, X^{\prime}$,  $Y,Y^{\prime}$, $\varepsilon, \delta$ are mutually independent. Define one-dimensional random variable
\begin{equation}\label{eq1}
 S= \delta \|X-Y\| -(1-\delta) \Bigl(\varepsilon\|X-X^{\prime}\|+(1-\varepsilon)\|Y-Y^{\prime}\|\Bigr).
\end{equation}

The following result holds.

\begin{thm}\label{th1}
Let the random objects $X, X^{\prime}$, $Y,Y^{\prime}$, $\varepsilon, \delta$ are as defined above. Then the random vectors $X$ and $Y$ are identically distributed if and only if $S$ has a symmetric distribution.
\end{thm}
\begin{proof}
Let $f_1(t)=\E \exp\bigl(it\|X-Y\|\bigr)$, $f_2(t)=\E \exp\bigl(it\|X-X^{\prime}\|\bigr)$, $f_3(t)=\E \exp\bigl(it\|Y-Y^{\prime}\|\bigr)$ be
characteristic functions of $\|X-Y\|$, $\|X-X^{\prime}\|$ and  $\|Y-Y^{\prime}\|$ respectively. It is easy to calculate characteristic function of $S$. It has the following form
\[ f(t) = \E\exp \bigl(itS\bigr) = \frac{1}{2}f_1(t) + \frac{1}{4}\Bigl(f_2(-t)+f_3(-t)\Bigr). \]
In the case of $X \stackrel{d}{=}Y$ we have
$f_1(t)=f_2(t)=f_3(t)$ and
\[ f(t) = \frac{1}{2}\Bigl(f_1(t)+f_1(-t)\Bigr), \]
so that the distribution of $S$ is symmetric around the origin.

Suppose now that the distribution of $S$ is symmetric around the origin and show that $X \stackrel{d}{=}Y$. In the case when this distribution is symmetric, then the mean of $\tt{sign}\,S\, \exp\{-S^2/2\}$ exists and equals to zero. However,
\[ \E\tt{sign}\,S\,\exp\{-S^2/2\} = \frac{1}{2}\E\exp\bigl(-\|X-Y\|^2/2\bigr)-\]
\[-\frac{1}{4}\Bigl(\E\exp\bigl(-\|X-X^{\prime}\|^2/2\bigr) +\E \exp\bigl(-\|Y-Y^{\prime}\|^2/2\bigr) \Bigr). \]
The equality of this expectation to zero implies that
\begin{equation}\label{eqN}
\E\exp\bigl(-\|X-Y\|^2/2\bigr)=
\end{equation}
\[= \frac{1}{2}\Bigl(\E\exp\bigl(-\|X-X^{\prime}\|^2/2\bigr) +\E \exp\bigl(-\|Y-Y^{\prime}\|^2/2\bigr) \Bigr). \]

It is known (see, for example \cite{Kl}), that the function
\[{\mc L}(x,y)= 1-\exp(-\|x-y\|^2/2) \]
is a strongly negative definite kernel on Euclidean or on Hilbert space. At the same time the following inequality holds
\[ \E {\mc L}(X,Y)-\frac{1}{2}\Bigl(\E {\mc L}(X,X^{\prime})+\E {\mc L}(Y,Y^{\prime})\Bigr) \geq 0. \]
The equality is attained if and only if  $X \stackrel{d}{=} Y$. It is clear that the relation (\ref{eqN}) guaranties the equality in previous inequality. 
\end{proof}

{\it Note that from the proof of Theorem \ref{th1} it follows that $X\stackrel{d}{=}Y$ if and only if}
\begin{equation}\label{eq2}
\|X-Y\| \stackrel{d}{=} \varepsilon\|X-X^{\prime}\|+(1-\varepsilon)\|Y-Y^{\prime}\|. 
\end{equation}

In other words, it means that multivariate hypothesis $X\stackrel{d}{=}Y$ may be changed by equivalent one-dimensional (\ref{eq2}).

Let us give previous result in term of an inequality.

\begin{thm}\label{thMI}
Let $X, X^{\prime}$, $Y,Y^{\prime}$, $\varepsilon, \delta$ be mutually independent random elements. Suppose that $X \stackrel{d}{=}X^{\prime}$, $Y \stackrel{d}{=}Y^{\prime}$, and the variables $\varepsilon$ и $\delta$ take the values $0$ and $1$ only with equal probabilities. The random vectors $X$ and $Y$ take values in finite-dimensional Euclidean space or in separable Hilbert space and have continuous distributions. Suppose that $h$ is a strictly convex continuously differentiable function on  $[0,1]$, $h(0)=0$. Denote by $F(x)$ cumulative distribution function of the random variable
\[ S= \delta \|X-Y\| -(1-\delta) \Bigl(\varepsilon\|X-X^{\prime}\|+(1-\varepsilon)\|Y-Y^{\prime}\|\Bigr) \]   
and set $G(x) = 1-F(-x)$. Then the following inequality holds
\begin{equation}\label{eqIM}
\int_{-\infty}^{\infty}h(F(x))dG(x)+\int_{-\infty}^{\infty}h(G(x))dF(x) \geq 2 \int_{0}^{1}h(u)du.
\end{equation}
The equality in (\ref{eqIM}) is attained if and only if $X \stackrel{d}{=}Y$.
\end{thm}
\begin{proof}
In our publication \cite{KV} there was given the following result.

{\it Let $h$ be strictly convex continuously differentiable function on $[0,1]$ such that $h(0)=0$. Suppose that $F(x)$ and $G(x)$ are continuous probability distribution functions (c.d.f.) on real line $\R^1$. 
Then
\begin{equation}\label{eqC}
\int_{-\infty}^{\infty}h(F(x))dG(x)+\int_{-\infty}^{\infty}h(G(x))dF(x) \geq 2 \int_{0}^{1}h(u)du
\end{equation}
with equality if and only if $F(x)=G(x)$.} 

It is clear that the statement of Theorem \ref{thMI} follows from this result and previous Theorem \ref{th1}.
\end{proof}

From above we see that from statistical point the testing of the hypothesis 
\[H_o: \quad X\stackrel{d}{=}Y \] 
in multidimensional case appears to be equivalent to
\[ \widetilde{H}_o: \quad \text{"Distribution of $S$ is symmetric around the origin".}\]
There are many  different tests for $\widetilde{H}_o$ and some of them are consistent.
However, in some cases the hypothesis is tested by different procedures. For example, one of such procedures is based on a property of median. Namely, if the distribution is symmetric around the origin then its median has to be zero. Therefore, if the median does not coincides with origin the hypothesis is rejected. This approach does not work in the case under consideration in view of the fact that median equals to zero for arbitrary distributed random vectors $X$ and $Y$.

It is clear that one may change the norm in construction of $S$ by arbitrary strongly negative definite kernel (see \cite{Kl} about terminology). This means, that it is possible to consider not only multidimensional Euclidean spaces or Hilbert spaces, but such metric spaces which allows isometric embedding into Hilbert space. We will consider some such spaces in a separate publication.

\section*{Acknowledgments} This work was partially supported by Grant GA\v{C}R 16-03708S. 
Authors are gratefull to Dr. K. Helisov\'{a} for her kind attention and useful advices.

\end{document}